\documentclass{amsart}
\usepackage{amssymb,latexsym}
\usepackage{graphicx, pgf, pst-all}
\usepackage{amsmath}
\usepackage{amsthm}
\usepackage{amsfonts}

\theoremstyle{plain}
\newtheorem{theorem}{Theorem}

\newtheorem{lemma}{Lemma}
\newtheorem{proposition}{Proposition}

\theoremstyle{definition}
\newtheorem{definition}{Definition}

\theoremstyle{remark}

\newtheorem{remark}[theorem]{Remark}

\numberwithin{equation}{section}
\def\rn{\mathbb{R}^{N}}
\def\capp{{\text{\rm cap}}_{p}}
\def\dive{{\rm div}}
\def\rife#1{(\ref{#1})}
\def\sob{W^{1,p}_{0}(\Omega)}
\newcommand{\elle}[1]{L^{#1}(\Omega)}
\newcommand{\cM}{\mathcal{M}}
\def\pw-1p'{L^{p'}(0,T;W^{-1,p'}(\Omega))}
\def\be{\begin{equation}}
\def\ee{\end{equation}}
\def\vare{\varepsilon}
\def\psob{L^{p}(0,T;W^{1,p}_{0}(\Omega))}
\def\dys{\displaystyle}
\def\al{\alpha}

\def\adixt#1{a(t,x,#1)}
\def\pmisure{\cM(Q)}

\def\rn{\mathbb{R}^{N}}

\def\be{\begin{equation}}
\def\ee{\end{equation}}
\def\rife#1{(\ref{#1})}
\def\re{\mathbb{R}}

\def\tku{T_{k}(u)}

\def\vare{\varepsilon}

\def\capp{\text{\text{cap}}_{p}}

\def\dive{{\rm div}}

\def\sob{W^{1,p}_{0}(\Omega)}

\def\into{\int_{\Omega}}

\def\w-1p'{W^{-1,p'}(\Omega)}

\def\w-1pd{W^{-1,p'}(D)}
\def\pw-1p'{L^{p'}(0,T;W^{-1,p'}(\Omega))}

\def\dys{\displaystyle}

\def\luq{L^{1}(Q)}
\def\lp'n{(L^{p'}(\Omega))^{N}}

\def\nuk{\nu^{k}}

\def\w-1p'{W^{-1,p'}(\Omega)}

\def\w-1pd{W^{-1,p'}(D)}

\def\w-1p'{W^{-1,p'}(\Omega)}

\def\m{\noalign{\medskip}}

\def\psob{L^{p}(0,T;W^{1,p}_{0}(\Omega))}

\def\luo{L^{1}(\Omega)}

\def\parelle#1{L^{#1}(Q)}
\def\vfi{\varphi}

\def\car#1{\raise2pt\hbox{$\chi$}_{#1}}

\def\adixt#1{a(t,x,#1)}

\newcommand{\rest}{\>\hbox{\vrule width.2pt \vbox to 7pt{\vfill\hrule width 7pt height.2pt}}\>}

\begin{document}


\title[Nonlinear Parabolic Equations with General Measure Data]
      {On the notion of renormalized solution to nonlinear parabolic equations with general 
measure data}


\author[F. Petitta]{Francesco Petitta}
\address[F. Petitta]{Dipartimento di Scienze di Base e Applicate
per l' Ingegneria, ``Sapienza", Universit\`a di Roma, Via Scarpa 16, 00161 Roma, Italy.} \email{francesco.petitta@sbai.uniroma1.it}
\author[A. Porretta]{Alessio Porretta} \address[A. Porretta]{Dipartimento di Matematica\\ Universit\`a di Roma Tor Vergata\\ Via della ricerca scientifica 1\\ 00133 Roma\\ Italy}
\email{porretta@mat.uniroma2.it}

\thanks{Received   May ..., Accepted }
\keywords{Nonlinear parabolic equations $\cdot$ Parabolic capacity $\cdot$ Measure data} 
\subjclass[2010]{35K55 $\cdot$ 35R06 $\cdot$  35R05}

\begin{abstract}

Here we introduce a new notion of renormalized solution  to nonlinear  parabolic problems with general measure data  whose model is 
$$
\begin{cases}
  u_t-\Delta_{p} u =\mu & \text{in}\ (0,T)\times\Omega,\\
 u=u_0  & \text{on}\ \{0\} \times \Omega,\\
 u=0 &\text{on}\ (0,T)\times\partial\Omega,
  \end{cases}
  $$
for any, possibly singular, nonnegative  bounded measure $\mu$. We prove existence of such a solutions and we discuss their main properties.  
\end{abstract}

\maketitle

\section{Introduction}
Let $\Omega$ be an open bounded subset of $\rn$, $T>0$, $p>1$, and let us consider the model problem 
 \begin{equation}\label{1}
\begin{cases}
  u_t-\Delta_{p} u =\mu & \text{in}\ (0,T)\times\Omega,\\
 u=u_0  & \text{on}\ \{0\} \times \Omega,\\
 u=0 &\text{on}\ (0,T)\times\partial\Omega,
  \end{cases}
  \end{equation}
where $\Delta_{p} u:= \dive(|\nabla u|^{p-2}\nabla u)$ is the $p$-Laplace operator and  $\mu$ is a bounded Radon measure on $Q=(0,T)\times\Omega$.

 If $\mu$ is a measure that does not charge sets of zero parabolic $p$-capacity (the so called \emph{diffuse measures}, see the definition in Section \ref{2.1} below), then a  notion of renormalized solution for problem \rife{1} was introduced in \cite{dpp}.  In the same paper the existence and uniqueness of such a  solution is proved. In \cite{DP}  a similar notion of entropy solution is also defined, and proved to be equivalent to the renormalized one.  The case in which  $\mu$ is a singular measure with respect to the $p$-capacity (i.e. $\mu$ admits a part that is concentrated on a set of zero $p$-capacity) was faced in \cite{pe2} if $p>\frac{2N+1}{N+1}$.   All these latest results are strongly based on a  decomposition theorem  given in \cite{dpp} for diffuse measures, the key point in the existence result  being the proof of the strong compactness of suitable truncations of the approximating solutions in the energy space.     

Recently,  in \cite{ppp} (see also \cite{ppp1}) the authors proposed a new approach to the same problem with diffuse measures as data. This approach  avoids to use the particular structure of the decomposition of the measure  and it seems more flexible  to handle a fairly general class of problems. In order to do that, the authors introduced a  definition of renormalized solution which is  closer to the one used for conservation laws in \cite{BCW} and to one of the existing formulations in the elliptic case (see \cite{DMOP} and \cite{DMM}).
 
Our goal is to extend the approach in \cite{ppp} to general, and possibly singular, measure data. In particular, we extend the notion of renormalized solution  given in \cite{ppp} to general measures and we prove an existence result in this framework. In order to avoid an excess of technicalities, for the sake of presentation we will deal with nonnegative data.  

The main advantage,  with this new approach, is that we do not need to pass through the usual key technical  step of the strong compactness of the truncations of the approximating solutions in order to prove the  existence result.  In fact, the possibility to prove existence of a solution by-passing the proof of strong convergence of truncations, which is  a heavy technical point, was already exploited in the stationary context in \cite{malu} using the particular structure of stationary diffuse measures. But, as we already mentioned, in the parabolic framework the situation is  more complicated due to the presence of the term $g_{t}$ in the decomposition formula (see \rife{gt} below) and we adopt a  different strategy, using  the approximation with equidiffuse measures as we already developed in \cite{ppp}.

Compared to our previous paper \cite{ppp}, we do not consider here zero  order terms in the equation since, as it is well known, in the case of singular data they produce, in general,  concentration phenomena and nonexistence results (see for instance \cite{bp, pe3}). Nevertheless, our existence result could easily be extended to the case of  lower order  nonlinearities with mild growth with respect to $u$ and $\nabla u$. 

\medskip
The paper is organized as follows. In section \ref{I} we give some preliminaries on the concept of $p$-capacity and on the functional spaces and the main notation we will use throughout the paper. Section \ref{II} will be devoted to set our main assumptions, to the definition of renormalized solution  and to the statement of the existence result, while in Section \ref{III}  we give the proof of our main result. In Section \ref{V} we finally discuss the relationship between the new approach and the previous ones and we extend the result to non-monotone operators.

\section{Preliminaries on capacity}\label{I}
\subsection{Parabolic $p$-capacity}\label{2.1}
The relevant notion in the study of problems as \rife{1} is the notion of  \emph{parabolic $p$-capacity}. 

We recall that for every $p > 1$ and every open subset
$U \subset Q$, the $p$-parabolic capacity of $U$ is given by (see \cite{pierre1,dpp})
$$
\capp(U)=\inf{ \Big\{\|u\|_{W} : u\in W,\ u \geq \chi_{U}\  \text{a.e. in}\  Q \Big\}},
$$
where 
$$
W=\big\{ u\in L^{p}(0,T;V) : u_{t}\in L^{p'}(0,T;V') \big\},
$$
being $V=\sob \cap \elle 2$ and $V'$  its dual space. As usual $W$ is endowed with the norm
$$
\|u\|_{W} =\|u\|_{ L^{p}(0,T;V)} + \|u_{t}\|_{ L^{p'}(0,T; V')}.
$$
The $p$-parabolic capacity $\capp$ is then extended to arbitrary Borel subsets $B \subset Q$ as
\[
\capp(B) = \inf{\Big\{ \capp(U) : B \subset U \text{ and } U \subset Q\ \text{is open}  \Big\}}.
\]

\vskip0.5em
\subsection{Diffuse measures and equidiffuse sequences}
Let $\cM(Q)$ denote the space of all bounded Radon measures on $Q$.  In the parabolic context this space is usually identified with the dual space of $C_{0}([0,T)\times\Omega )$, the space of all continuous functions that vanish at the parabolic boundary  $(0,T]\times\partial \Omega $.  Henceforth, we call a finite measure $\mu$ \emph{diffuse}\/ if it does not charge sets of zero $p$-parabolic capacity, i.e.\@ if $\mu(E) = 0$ for every Borel set $E \subset Q$ such that $\capp(E) = 0$.  The subspace of all diffuse measures in $Q$ will be denoted by $\cM_0(Q)$.

According to a representation theorem  proved in \cite{dpp}, for every $\mu \in \mathcal{M}_0(Q)$ there exist $f\in L^{1}(Q)$, $g\in
L^{p}(0,T;V)$ and $\chi\in \pw-1p'$ such that
\be\label{gt}
\mu = f + g_t + \chi \quad \text{in } \mathcal{D'}(Q).   
\ee

The presence of the term $g_t$ in the decomposition of a diffuse measure, that  is essentially due to the presence of diffuse measures which charges sections of the parabolic cylinder $Q$, gives some extra difficulties (with respect to the stationary case) in the study of this type of problems; in particular the proof of the strong convergence of suitable truncations of the approximating solutions is a hard technical issue.  For further considerations on this fact we refer to \cite{pe2}, \cite{ppp}, \cite{bvq} and references therein.   
\smallskip

As we said, in order to avoid those difficulties we adopt a  different strategy, that is essentially independent of the decomposition of the measure data. A crucial role will be played by  an important property enjoyed by the convolution of diffuse measures. We recall the following definition  (see \cite{BrePon:05a} and also \cite{ppp}):

\smallskip
\begin{definition}
A sequence of measures $(\mu_n)$ in $Q$ is \emph{equidiffuse} if for every $\vare>0$ there exists $\eta>0$ such that
\[
\capp(E)<\eta \quad \Longrightarrow \quad |\mu_n|(E)<\vare \quad \forall n\geq 1. 
\]  
\end{definition}
\smallskip

Let $\rho_n$ be a sequence of mollifiers on $Q$. The following result is proved in \cite{ppp}.
\begin{proposition}\label{equi}
If $\mu \in \cM_0(Q)$, then the sequence $(\rho_n \ast\mu)$ is equidiffuse.  
\end{proposition}

\subsection{General measures and generalized gradient}Henceforward, we will say that a sequence $\{\mu_n \}\subset \mathcal{M}(Q)$ converges \emph{tightly} (or, equivalently, in the \emph{narrow topology of measures}) to a measure $\mu$ if
$$
\lim_{n\to\infty}\int_Q\varphi\ d\mu_n = \int_Q\varphi\ d\mu, \ \ \ \forall \ \varphi\in C(\overline{Q}).
$$

We point out that, at least for nonnegative measures,  tight convergence is equivalent to $\ast$-weak convergence provided the masses converge: that is,  $\mu_n$ converges tightly to $\mu$ if and only if $\mu_n$ converges to $\mu$  $\ast$-weak in $\mathcal{M}(Q)$ and $\mu_n (Q)$ converges to $\mu(Q)$. 
Via a standard convolution argument one can prove the following
\begin{lemma}\label{appc1}
Let $\mu\in \mathcal{M}(Q)$. Then there exists a sequence $\{\mu_n\}\subset C^{\infty}(Q)$ such that 
$$
\|\mu_n \|_{L^1 (Q)}\leq |\mu|_{\mathcal{M}(Q)}\,,
$$
and
$$
\mu_n \longrightarrow \mu \ \ \text{tightly in }\ \mathcal{M}(Q).
$$
\end{lemma}

We define the restriction of a measure $\nu$ on a Borel set $E\subset Q$ as 
$
\nu\rest E (B)= \nu(B\cap E)\,,\  \text{for any}\  B\subset Q\,.
$ 
We say that a measure $\nu$ is concentrated on a Borel set $E$ if $\nu\rest E=\nu$.  
If $\mu$ is a bounded measure in $\mathcal{M}(Q)$ then we consider its decomposition into  diffuse and singular parts, that is
$$
\mu=\mu_d +\mu_s,
$$
 where $\mu_d$ is a measure in $\mathcal{M}_0(Q)$, that is $\mu_d$ is absolutely continuous with respect to the the $p$-capacity, while $\mu_s$ is concentrated on a set of zero $p$-capacity.

\vskip0.5em

A classical feature for problems with irregular data is that solutions typically turn out to not belong to the usual energy space, and not even to any  Sobolev space if $p$ is close  to $1$. Because of that, let us  precise  what we mean by $\nabla u$ even if $u$ may not belong to any Sobolev space. We follow the definition of generalized gradient introduced in \cite{B6} for functions $u$ whose truncations belong to a  Sobolev space. For $s$ in $\mathbb{R}$, and $k > 0$, we will use the standard truncation at levels $\pm k$ defined by $T_{k}(s) = \max(-k,\min(s,k))$. We have the following

\begin{definition}
Let $u\,:\,Q\to \mathbb R$ be a measurable function which is almost everywhere finite and such that $T_k(u)\in \psob $ for every $k>0$. Then (see \cite{B6}, Lemma 2.1) there exists a unique vector-valued function $U$ such that
$$
U= \nabla T_k(u)\,\chi_{\{|u|<k\}}\quad \hbox{a.e. in  $Q$}\,,\qquad \forall\,\,  k>0\,.
$$ 
This function $U$ will be called the gradient of $u$, hereafter denoted by $\nabla u$. When $u\in L^1(0,T;W^{1,1}_0(\Omega))$, it coincides with the usual distributional gradient.
\end{definition}

Finally,  we will use the following notation for sequences:
 $\omega (h, n, \delta, ...)$ will  indicate any quantity that vanishes as the parameters go to
their (obvious, if not explicitly stressed) limit point,  with the same order in which they appear,  that is, for instance
$$
\dys\lim_{\delta\rightarrow 0} \limsup_{n\rightarrow +\infty}
\limsup_{h \to 0} |\omega(h,n,\delta)|=0.
$$

\section{Main assumptions and renormalized formulation} \label{II}


 Let $p>1$ and assume that $a \,:\,Q
\times \rn \to \rn$ is a Carath\'eodory function (i.e., $a(\cdot,\cdot,\xi)$
is measurable on $Q$ for every $\xi$ in $\rn$, and $a(t,x,\cdot)$ is
continuous on $\rn$ for almost every $(t,x)$ in $Q$), such that the
following holds:
\be
a(t,x,\xi)  \xi \geq \al\,|\xi|^p\,,
\label{coercp}
\ee
\be
|a(t,x,\xi)| \leq \beta\,[b(t,x) + |\xi|^{p-1}]\,,
\label{cont}
\ee
\be
[a(t,x,\xi) - a(t,x,\eta)] (\xi - \eta) > 0\,,
\label{monot}
\ee
for almost every $(t,x)$ in $Q$, for every $\xi$, $\eta$ in $\rn$, with
$\xi \neq \eta$, where
$\al$ and $\beta$ are two positive constants, and
$b$ is a nonnegative function in $L^{p'}(Q)$.

We consider the initial boundary value problem
\begin{equation}\label{renpb}
\begin{cases}
    u_t-\dive(\adixt {\nabla u} )=\mu & \text{in}\ Q,\\
    u(0,x)=u_0  & \text{in}\ \Omega,\\
 u(t,x)=0 &\text{on}\ (0,T)\times\partial\Omega,
  \end{cases}
\end{equation} 
where $\mu$ is  a nonnegative Radon measure on $Q$ such that $|\mu|(Q)<\infty$ and $u_0\in \elle1$ is a nonnegative function.

\smallskip 

The following definition is the natural extension of the one given in \cite{ppp} for diffuse measures (see also \cite{pe2}). 

\begin{definition}\label{rendef}
 A function $u\in L^1(Q)$ is a renormalized solution of problem \rife{renpb} if $T_k(u)\in \psob$ for every $k>0$ and if there exists a sequence of  nonnegative measures $\nu^k\in \pmisure$ such that:
\be\label{limnu}
 \nuk \longrightarrow \mu_s \ \   \text{tightly} \ \ \text{as}\ k\to+\infty\,,
\ee
and
\be\label{req}
\begin{array}{c}
\dys -\int_Q T_k(u)\,\vfi_t\,dxdt +\int_Q \adixt {\nabla T_k(u)}\nabla \vfi\,dxdt=
\\
\m
\qquad
\dys \int_Q \vfi\,d\mu_d +\int_Q \vfi\,d\nu^k+\into T_k(u_0)\vfi(0)\,dx 
\end{array}
\ee
for every $\vfi\in C^\infty_c([0,T)\times \Omega)$.
\end{definition}

\begin{remark}
Some considerations are in order concerning 
Definition \ref{rendef}. First of all, observe that \rife{req} implies that $T_k(u)_t-\dive(\adixt{\nabla T_k(u))}$ is a bounded measure, and since $T_k(u)\in \psob$ this means that
$$
T_k(u)_t-\dive(\adixt{\nabla T_k(u)})\in W'\cap \pmisure\,.
$$
In particular, we have
$$
T_k(u)_t-\dive(\adixt{\nabla T_k(u)})=\mu_d+\nu^k\qquad \hbox{in $\pmisure$.}
$$
In view of Proposition 3.1 in \cite{ppp}, then $\nu^k$  is a diffuse measure. This is a key fact since it allows us  to recover from equation \rife{req} the standard estimates known for nonlinear potentials. 

Moreover, if $\mu$ is diffuse then Definition \ref{rendef} coincides with Definition 4.1 in \cite{ppp}. This fact is easy to check once we observe that nonnegative measures that vanish  tightly actually strongly converge  to zero in $\mathcal{M}(Q)$. 
\end{remark}

\medskip

First of all, it is possible to consider a  larger class of test functions.  

\begin{proposition}\label{testW} Let $u$ be a renormalized solution in the sense of Definition \ref{rendef}. Then we have
$$
\begin{array}{c}
-\dys\int_Q T_k(u)\,v_t\,dxdt +\int_Q \adixt {\nabla T_k(u)}\nabla v\,dxdt=
\\
\m
\qquad
\dys\int_Q \tilde v\,d\mu_d+\int_Q \tilde v\,d\nu^k+\into T_k(u_0)v(0)\,dx 
\end{array}
$$
for every $v\in W\cap\parelle\infty$ such that $v(T)=0$ (with $\tilde v$ being the unique cap-quasi continuous representative of $v$).
\end{proposition}
\begin{proof} 
Since $\nu_k$ is diffuse we get the result reasoning as in \cite[Proposition 4.2]{ppp}\,. 
\end{proof}
\vskip1em

Proposition \ref{testW} essentially  allows us to use test functions that depend on the solution itself in \rife{req}.  
Then, reasoning exactly as in \cite[Proposition 4.5]{ppp}, renormalized solutions can be proved to be distributional solutions and to enjoy some basic a priori estimates.

\begin{proposition}
Let $u$ be a  renormalized solution of \rife{renpb}. Then $u$ satisfies, for every $k>0$ and $\tau\leq T$:
$$
\into \Theta_k (u)(\tau) \, dx + \int_{0}^{\tau} \!\! \into |\nabla \tku |^p \, dx\, dt \leq C\, k\left( \|\mu\|_{\mathcal{M}(Q)}+ \|u_0\|_{\elle1}\right),
$$
where $\Theta_k(s)= \int_0^s T_k(t)dt$.

\noindent Therefore,  $u\in L^\infty(0,T;\elle1)$, $| \nabla u |^{p-1}$ and $a(t,x,\nabla u) \in L^{r}(Q)$ for any $r<\frac{N+p'}{N+1}$. Moreover, $u$ is a distributional solution, that is 
$$
\begin{array}{c}
-\dys\int_Q\, u \varphi_t\,dxdt +\int_Q \adixt {\nabla u}\nabla \varphi\,dxdt=
\dys\int_Q  \varphi \,d\mu\,,
\end{array}
$$
for any $\varphi\in C_{c} ([0,T)\times\Omega)$ and $u(0,x)=u_{0}$ in the sense of $L^{1}(\Omega)$. 
\end{proposition}

Now we can state our existence result
\begin{theorem}\label{exis}
Let $\mu\in\mathcal{M}(Q)$ be a nonnegative measure and $0\leq u_0\in \luo$. Then there exists a renormalized solution to problem \rife{renpb}. 
\end{theorem}

\section{Proof of Theorem \ref{exis}}\label{III}

In this section we prove Theorem \ref{exis}.  As usual for nonlinear equations with measure data, we  will prove the existence of solutions  through  approximation of the data $\mu$  with smooth functions. Thus, let $\mu_n= (\rho_n \ast\mu)$  and let $u_{n,0}$ be a sequence of functions in $C_c(\Omega)$ that converge to $u_0$ in $\luo$, and consider the following approximation problem
\begin{equation}\label{senzah}
\begin{cases}
    (u_{n})_{t}-\dive(a(t,x,\nabla  u_n)) =\mu_n & \text{in}\ Q,\\
    u_n(0,x)=u_{n,0} & \text{in}\ \Omega,\\
 u_n(t,x)=0 &\text{on}\ (0,T)\times\partial\Omega. 
  \end{cases}
\end{equation}  
The existence of a nonnegative weak solution for problem \rife{senzah} is classical (see for instance \cite{l})

\smallskip
We will need the following basic compactness result which are nowadays classical.
\begin{proposition}\label{pro}
Let $u_n$ be the sequence of solutions for  problem \rife{senzah}.  Then
$$
\|u_n\|_{L^\infty (0,T; L^1 (\Omega))}\leq C,
$$
$$
\int_Q|\nabla T_k (u_n)|^p\ dxdt\leq Ck \qquad \forall k>0\,.
$$
Moreover, there exists a measurable function $u$ such that $T_k (u)\in\psob$ for any $k>0$, $u\in L^\infty (0,T; L^1 (\Omega))$, and, up to  a subsequence,  we have
$$
\begin{array}{l}
u_n\longrightarrow u \qquad \hbox{a.e. in $Q$    and strongly in $L^1 (Q)$,}\\\\
T_k (u_n)\rightharpoonup T_k (u)\qquad \hbox{weakly in $ \psob$  and a.e. in $Q$,}\\\\
\nabla u_n\longrightarrow \nabla u \qquad \hbox{a.e. in $Q$.}\\\\
|\nabla u_n|^{p-2}\nabla u_n\longrightarrow\ \ |\nabla u|^{p-2}\nabla u\qquad  \hbox{ in $\luq$.}
\end{array}
$$
\end{proposition}

\proof See \cite{bdgo}, \cite{dpp}, \cite{pe2}, \cite{po}, \cite{ppp}.
\qed

\vskip1em

A key tool in our analysis is contained in the following  result proved in \cite{ppp}

\begin{theorem}\label{stimcap}
Let $\mu$ be a nonnegative measure in $\mathcal{M}(Q)\cap\pw-1p'$ and $0\leq u_0\in \elle2$, let $u\in W$ be the solution of
$$
\begin{cases}
    u_t-{\rm div}(a(t,x,\nabla u))=\mu & \text{in}\ Q,\\
u=u_0  & \text{on}\ \{0\} \times \Omega,\\
u=0 &\text{on}\ (0,T)\times\partial\Omega.
  \end{cases}
$$
Then, 
$$
\capp(\{u > k\})\leq C\max\left\{\frac{1}{k^{\frac{1}{p}}},\frac{1}{k^\frac{1}{p'}}\right\} \quad \forall k\ge 1,
$$
where $C >0$ is a constant depending on $\|\mu\|_{\mathcal{M}(Q)}$, $\|u_0\|_{\elle1}$ and $p$.
\end{theorem}

In order to work separately  either near to and far from the singular set of $\mu$ we also need to construct suitable cut-off functions. Let us consider the space 
$$
S=\{ z\in L^p (0,T;V); z_t \in L^{p'}(0,T;\w-1p' ) +  \luq \}
$$
endowed with its norm
$$
\| z\|_S= \|z\|_{L^p (0,T;V)}+ \|z_t\|_{L^{p'}(0,T;\w-1p' ) +  \luq }.
$$ 
Then we have the following technical result was proof can be obtained as in \cite{pe2}. 
\begin{lemma}\label{cut}
Let $\mu_s$ be a nonnegative bounded Radon measure concentrated on a set $E$ of zero $p$-capacity. Then, for any $\delta>0$, there exists a compact set $K_\delta\subseteq E$ and a function $\psi_\delta \in C^{\infty}_{c}(Q)$ such that 
$$
\mu_s(E\backslash K_\delta)\leq \delta,  \ \ 0\leq \psi_\delta\leq 1, \ \ \psi_\delta\equiv 1 \ \ \text{on $K_\delta$}, 
$$
and
$$
\psi_\delta \to 0 \ \ \  \text{in $S$ \  as $\delta\to 0$}. 
$$
Moreover, 
$$
\int_Q (1-\psi_\delta) \ d\mu_s=\omega(\delta)
$$
\end{lemma}

%

\begin{proof}[Proof of Theorem \ref{exis}]

The proof of Theorem \ref{exis} will be derived in a few steps. First of all, as we said,  let  
$$
\mu_n=  \rho_n \ast \mu=  \rho_n \ast \mu_d+ \rho_n \ast \mu_{s}^{}:=\mu_d^n+\mu_{s}^{n}. 
$$
Observe that, based on Proposition \ref{equi} then $\mu_d^n=\rho_n \ast \mu_d$ is an equidiffuse sequence of measures. Moreover, $\mu_n$ satisfies the properties of  Lemma \ref{appc1}.
 
 We also define, for fixed $\sigma>0$ 
$$
\dys S_{k,\sigma} (s)=
\begin{cases}
1 & \text{if } s\leq k,\\
0&\text{if }s > k+\sigma,\\
\mathrm{affine} & {\rm otherwise}, 
\end{cases}
$$
and   $h_{k,\sigma}(s):=1- S_{k,\sigma}(u_n)$.

\begin{figure}[!ht]
\centerline{
\begin{picture}(200,80)(-100,-35)
\put(-100,0){\vector(1,0){200}}
 \put(0,-35){\vector(0,1){80}}
\thicklines
   \put(-100,20){\line(1,0){130}}
 \put(30,20){\line(3,-2){30}}
 \put(60,0){\line(1,0){40}}
\thinlines
 \multiput(30,0)(0,7){3}{\line(0,1){3}}
 \put(28,-10){$\scriptstyle k$}
 \put(60,-10){$\scriptstyle k+\sigma$}
 \put(-8,24){$\scriptstyle 1$}
 \put(92,-10){$\scriptstyle s$}
 \put(4,38){$\scriptstyle S_{k,\sigma} (s)$}
\end{picture}
}\caption{The function $S_{k,\sigma} (s)$}
\end{figure}
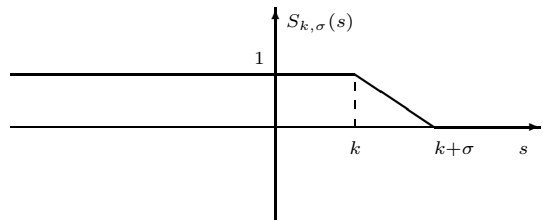

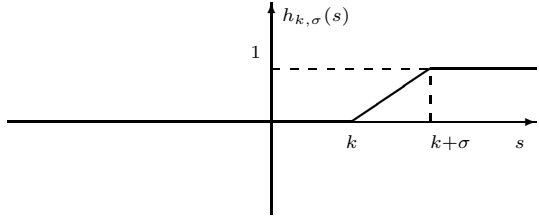
\begin{figure}[!ht]
\centerline{
\begin{picture}(200,80)(-100,-35)
\put(-100,0){\vector(1,0){200}}
 \put(0,-35){\vector(0,1){80}}
\thicklines
   \put(-100,0){\line(1,0){130}}
 \put(30,0){\line(3,2){30}}
 \put(60,20){\line(1,0){40}}
\thinlines
 \multiput(60,0)(0,7){3}{\line(0,1){3}}
 \multiput(0,20)(7,0){9}{\line(1,0){3}}
 \put(28,-10){$\scriptstyle k$}
 \put(60,-10){$\scriptstyle k+\sigma$}
 \put(-8,24){$\scriptstyle 1$}
 \put(92,-10){$\scriptstyle s$}
 \put(4,38){$\scriptstyle h_{k,\sigma} (s)$}
\end{picture}
}\caption{The function $h_{k,\sigma} (s)$}
\end{figure}

 \medskip
{\it Step 1. Estimates in $L^1(Q)$ on the energy term.} For fixed $\sigma>0$ 
  we take $h_{k,\sigma}(u_n)$ in \rife{senzah},  $H_{k,\sigma}(s)=\int_{0}^{s}h_{k,\sigma}(\eta)d\eta$,  to obtain
$$
\begin{array}{l}
 \dys\int_\Omega H_{k,\sigma}(u_n(T)) + \frac{1}{\sigma}\int\limits_{\{k<u_n<k+\sigma\}} a(t,x,\nabla u_n)\nabla u_n \\\\ \dys
= \int_Q \mu^n h_{k,\sigma}(u_n) + \int_\Omega H_{k,\sigma}(u_{0,n}),  
\end{array}
$$
so that, dropping positive terms 
\begin{multline*}
\frac1\sigma\int\limits_{\{k<u_n<k+\sigma\}}\adixt{\nabla u_n}\nabla u_n\,dxdt
\leq \int\limits_{\{u_{n}>k\}} \mu_n\,dxdt
+ \int\limits_{\{u_{0n}>k\}} u_{0n}\,dx,
\end{multline*}
which implies, in particular,
$$
\frac{1}{\sigma}\int_{\{k<u_n\leq k+\sigma\}}a(t,x,\nabla u_n)\nabla u_n\leq C.
$$
Thus, there exists a bounded Radon measure $\lambda_k^n $ such that, as $\sigma$ goes to zero 
 $$
 \frac{1}{\sigma}a(t,x,\nabla u_n)\nabla u_n \chi_{\{k<u_n\leq k+\sigma\}}\rightharpoonup \lambda_k^n\ \ \ \text{$\ast$-weakly in }  \mathcal{M}(Q). 
 $$
 
 \medskip
{\it Step 2. Equation for the truncations.} Now we are able to prove that \rife{req} holds true for $u$. 
To see that, we multiply \rife{senzah} by $S_{k,	\sigma}(u_n)\xi$, where $\xi\in C^{\infty}_c ([0,T)\times\Omega)$ to  obtain, after taking the limit as $\sigma$ vanishes
$$
\begin{array}{l}
\dys T_k (u_n)_t -{\rm div}(a(t,x,\nabla T_k(u_n))) -\mu_d^n   \dys=\lambda_k^n\dys+ \mu_{s}^{n}\chi_{\{u_n\leq k\}}- \mu_{d}^{n}\chi_{\{u_n\geq k\}},
\end{array}
$$
in $\mathcal{D}'(Q)$. 
We define the measure $\nu_n^k$ as
$$
\nu^k_n:=\lambda_k^n+ \mu_{s}^{n}\chi_{\{u_n\leq k\}}- \mu_{d}^{n}\chi_{\{u_n\geq k\}}\,.
$$
Notice that
$$
\|\nu^k_n\|_{\luq}\leq C
$$

so that there exists $\nu^k\in \cM (Q) $ such that 
$$
\nu^k_n\rightharpoonup \nu^k
$$
$\ast$-weak in $ \cM (Q)$.

Therefore, using Proposition \ref{pro}, in the sense of distributions  we have
\begin{equation}\label{eqtk}
T_k (u)_t -{\rm div}(a(t,x,\nabla T_k(u))) =\mu_d  +\nu^k.
\end{equation}

\medskip
{\it Step 3. The limit of $\nu^k$. }
Let us consider the distributional formulation  of \rife{senzah} and let us subtract \rife{eqtk} from it, to obtain, for any $\xi\in C^\infty_c ([0,T)\times\Omega)$
$$
-\int_Q (u_n - T_k(u))\xi_t + \int_Q (a(t,x,\nabla u_n)- a(t,x,\nabla T_k(u)))\nabla\xi
$$
$$
=\int_Q\xi d(\mu_d^n - \mu_d) +\int_Q \xi d(\mu_s^n - \nu^k) + \into \xi(0)(u_{n,0}- T_k(u_0)) \, . 
$$  
Using Proposition \ref{pro} we can pass to the limit with respect to $n$ to obtain 
$$
\nu^k =\mu_s +\omega(n,k)
$$
in $\mathcal{D}'(Q)$.

\medskip
To complete the proof we have to show that the previous limit is actually tight. Let us choose without loss of generality $\varphi\in C^1(\overline{Q})$ (then an easy density argument will show that the result holds with $\varphi\in C(\overline{Q})$). We have
 $$
 \begin{array}{l}
\dys \int_Q\varphi d\nu^k = \int_Q\varphi \psi_\delta d\nu^k +\int_Q\varphi (1-\psi_\delta)d\nu^k, 
 \end{array}
 $$
 where $\psi_{\delta}$ is chosen as in Lemma \ref{cut}. 
Thanks to the previous result we have 
\be\label{quiz}
\dys \int_Q\varphi \psi_\delta d\nu^k=\int_Q \varphi  \psi_\delta d\mu_s +\omega(k)\,. 
\ee
Recalling that $\psi_\delta=1 $ on $K_\delta$, we have 
$$
\dys\int_Q \varphi  \psi_\delta d\mu_s= \int_{K_{\delta}} \varphi  d\mu_s +\int_{E\backslash K_{\delta}} \varphi  \psi_\delta d\mu_s \,.
$$
Now, using Proposition \ref{cut} we get both 
$$
\int_{E\backslash K_{\delta}} \varphi  \psi_\delta d\mu_s  \leq  \delta\|\varphi\|_{L^{\infty}(Q)}
 $$
and (by Lebesgue's theorem)  
$$
\int_{K_{\delta}} \varphi  d\mu_s = \int_{Q} \varphi  d\mu_s+\omega (\delta)\,,
$$
that gathered together with \rife{quiz} gives
$$
\int_Q\varphi \psi_\delta d\nu^k=\int_Q \varphi   d\mu_s +\omega(k,\delta)\,. 
$$
{\it Step 4. Proof Completed.}
To conclude we have to prove that 
$$
\int_Q\varphi (1-\psi_\delta)d\nu^k=\omega(k,\delta). 
$$

From the definition of $\nu^k$ we have  that
$$
 \begin{array}{l}
\dys\int_Q\varphi (1-\psi_\delta)d\nu^k=\lim_n\left(\lim_\sigma \frac{1}{\sigma}\int_{\{k<u_n\leq k+\sigma\}}a(t,x,\nabla u_n)\nabla u_n \varphi(1-\psi_\delta)\right.\\\\
\dys\left. +\int_{\{u_n\leq k\}}\varphi(1-\psi_\delta) d\mu_s^n -\int_{\{u_n> k\}}\varphi(1-\psi_\delta) d\mu_d^n\right). 
 \end{array}
 $$
 Now, using Theorem \ref{stimcap} and recalling that $\mu_d^n$ are equidiffuse, we get
 $$
 \int_{\{u_n> k\}}\varphi(1-\psi_\delta) d\mu_d^n=\omega(n,k)\,.
 $$
Moreover, using Lemma \ref{cut} we get
$$
\left|\int_{\{u_n\leq k\}}\varphi(1-\psi_\delta) d\mu_s^n\right|\leq\|\varphi\|_{L^{\infty}(Q)}\int_Q (1-\psi_\delta) d\mu_s^n=\omega(n,\delta). 
$$
The proof is complete once we prove
\begin{equation}\label{ultima}
\frac{1}{\sigma}\int_{\{k<u_n\leq k+\sigma\}}a(t,x,\nabla u_n)\nabla u_n \varphi(1-\psi_\delta)=\omega(\sigma,n,k,\delta). 
\end{equation}
To do that, we use again \rife{senzah} with $h_{k,\sigma}(u_n)(1-\psi_\delta)$ as test function to obtain
$$
\begin{array}{l}
 \dys\int_Q H_{k,\sigma}(u_n(t,x))(\psi_\delta)_t - \into H_{k,\sigma}(u_{n,0}(x))(1-\psi_\delta(0)) \\
 \\
 \dys 
 + \frac{1}{\sigma}\int\limits_{\{k<u_n<k+\sigma\}} a(t,x,\nabla u_n)\nabla u_n (1-\psi_\delta) - \int_Q a(t,x,\nabla u_n)\nabla \psi_\delta h_{k,\sigma}(u_n)  
\\\\ \dys = \int_Q \mu_d^n h_{k,\sigma}(u_n) (1-\psi_\delta)+ \int_Q \mu_s^n h_{k,\sigma}(u_n) (1-\psi_\delta). 
\end{array}
$$
Using the convergence in $L^1(Q)$ of $u_n$ and $|a(t,x,\nabla u_n)|$, and the regularity of $\psi_\delta$ we easily get 
$$
\int_Q H_{k,\sigma}(u_n(t,x))(\psi_\delta)_t =\omega(n,k), \ \ \text{and} \ \  \int_Q a(t,x,\nabla u_n)\nabla \psi_\delta h_{k,\sigma}(u_n)=\omega(n, k).
$$
Similarly we get rid of the term at $t=0$. 
Moreover, thanks to Theorem \ref{stimcap} and the equidiffuse property of $\mu_d^n$, 
$$
\left| \int_Q \mu_d^n h_{k,\sigma}(u_n) (1-\psi_\delta)\right|\leq \int_{\{u_n>k\}} \mu_d^n (1-\psi_\delta) =\omega(n,k)\,.
$$
Finally, we have
$$
\left|\int_Q \mu_s^n h_{k,\sigma}(u_n) (1-\psi_\delta)\right|\leq \int_Q \mu_s^n  (1-\psi_\delta)=\omega(n,\delta)
$$
where we used Lemma \ref{cut} in the last equality. Gathering together all these facts we get \rife{ultima}. 
\end{proof}

\section{Some further properties and remarks}\label{V}

\subsection{An asymptotic reconstruction property}
As we have seen, in this paper we provide a different, and in some sense more natural, approach to nonlinear parabolic problems with measure data. One of the main points is that we  do not pass through the strong convergence of the truncations. We stress again that this is not only a technical point: in fact, due to the presence of the time derivative part $g_{t}$ in the decomposition of the measure $\mu$, the strong compactness of $T_{k}(u_{n})$ in the energy space is not known in general. What can be proven in many cases is that truncations of suitable translations of the approximating solutions are strongly compact in the right energy space (see \cite{dpp} and \cite{pe2} for further remarks on this fact). 

Our approach by-passes this problem. Anyway, to be consistent with the literature, we want to stress how  our analysis fits with the previous studies.   

{ In \cite{pe2} the role of  property \rife{limnu} in Definition \ref{rendef} is essentially played by a reconstruction property of the singular part of the measure. Indeed, one expects that $\mu_s$ can be obtained asymptotically  as
\be\label{asy}
\lim_{h\to\infty}\int_{\{h-1\leq u<h\}} a(t,x,\nabla u)\nabla u \xi\ dt dx=\int_Q\xi \ d \mu_s
\ee
for any $\xi\in C^{\infty}_{c}([0,T)\times\Omega)$. This property  is known to hold  for any renormalized solution  in the stationary case (\cite{DMOP}). In the evolution case,  it was proved to hold in \cite{pe2}  when the measure $\mu$  has no diffuse part; on the other hand, for the case of general measure, 
the property  was only proved to hold for some translation of $u$ (depending on the decomposition of $\mu_d$) and not for $u$ itself. }

Here we want to emphasize how this type of property is essentially contained in our definition. In fact we prove that  the singular part of the measure $\mu$   is reconstructed by the energy of the approximating solutions on their own  level sets. 
Namely, we have the following.

\begin{proposition}\label{rico}
Let $u_n$ be solution of \rife{senzah}, then
$$
 \lim_{h}\limsup_{n} \int_{\{h-1\leq u_n<h\}} a(t,x,\nabla u_n)\nabla u_n \ \xi= \int_Q \xi \ d\mu_s\,,
$$
  for any $\xi\in C^{\infty}_{c}([0,T)\times\Omega)$.
\end{proposition}

We need the following result, which turns out to have its own interest, that  shows how the approximating solutions behave around the singular sets where $\mu$ is concentrated.  

\begin{lemma} Let $u_n$ be solution of \rife{senzah}, $k>0$,  and let $\psi_\delta$ be as in Lemma \ref{cut},  then
\begin{equation}\label{stima1}
\int_{Q} \mu_s^{n}(k-u_n)^{+} \psi_{\delta}^{}=\omega(n, \delta). 
\end{equation}

\end{lemma}

\begin{proof}
We multiply the equation in \rife{senzah} by $(k-u_n)^+\psi_\delta$ where $\psi_\delta$ is given by Lemma \ref{cut} and we integrate over $Q$. We get 
$$
\begin{array}{l}
 \dys-\int_Q (\int_0^{u_n} (k-s)^+ ds)(\psi_\delta)_t   + \int_Q a(t,x,\nabla u_n)\nabla \psi_\delta (k-u_n)^+ \\\\ \dys
= \dys \int_Q a(t,x,\nabla T_k (u_n))\nabla T_k(u_n)\psi_\delta + \int_Q \mu^n(k-u_n)^+ \psi_\delta + \into(\int_0^{u_{n,0}} (k-s)^+ ds) \psi_\delta(0)
\end{array}
$$
 Now, using Proposition \ref{pro}, observing that  $(\int_0^{u} (k-s)^+ ds)\in \psob\cap L^\infty(Q)$, and that $\psi_\delta$ goes to zero in $S$, we get both
 $$
 -\int_Q (\int_0^{u_n} (k-s)^+ ds)(\psi_\delta)_t  =\omega (n,\delta),
 $$
and 
$$
 \int_Q a(t,x,\nabla u_n)\nabla \psi_\delta (k-u_n)^+= \int_Q a(t,x,\nabla T_{k}(u_n))\nabla \psi_\delta (k-u_n)^+=\omega (n,\delta). 
$$
So that, dropping the nonnegative  terms in the right-hand side, we deduce \rife{stima1}.  Let us also observe that,  as a by-product, we also have the following property of the energy of the truncations near the singular set:
$$
\alpha\int_Q |\nabla T_k (u_n)|^p\psi_\delta\leq \int_Q a(t,x,\nabla T_k (u_n))\nabla T_k(u_n)\psi_\delta=\omega (n,\delta). 
 $$
 \end{proof}

\begin{proof}[Proof of Proposition \ref{rico}]
 Let 
$$
\dys \Theta_{h} (s)=
\begin{cases}
1 & \text{if } s\geq h,\\
0&\text{if }s < h-1,\\
\mathrm{affine} & {\rm otherwise}, 
\end{cases}
$$
 and let us take $\Theta_{h} (u_n)\xi$ as test function in \rife{senzah}, where $\xi\in C^{\infty}_{c} ([0,T)\times\Omega)$, to have 
 $$
\begin{array}{l}
 \dys-\int_Q (\int_0^{u_n} \Theta_{h} (s)ds)\xi_t  + \int_{\{h-1\leq u_n<h\}} a(t,x,\nabla u_n)\nabla u_n \ \xi\\\\ \dys  + \int_Q a(t,x,\nabla u_n)\nabla\xi \Theta(u_n) 
= \dys \int_Q \xi \Theta (u_n)\ d\mu_d^n + \int_Q \xi \Theta (u_n)\ d\mu_s^n + \into (\int_0^{u_{n,0}} \Theta_{h} (s)ds)\xi(0)\,.
\end{array}
$$
 Let us analyze the previous terms one by one. First of all, thanks to Proposition \ref{pro} we easily have 
 $$
 -\int_Q (\int_0^{u_n} \Theta_{h} (s)ds)\xi_t=\omega (n,h), 
 $$
 and 
 $$
 \int_Q a(t,x,\nabla u_n)\nabla\xi \Theta(u_n) =\omega (n,h).  
 $$
 Similarly we get rid of the term at $t=0$. 
 Moreover, using the fact that $\mu^n_d $ is equidiffuse and Theorem \ref{stimcap} we have 
 $$
  \int_Q \xi \Theta (u_n)\ d\mu_d^n\leq C\int_{\{u_n\geq h-1\}}d\mu_d^n=\omega(n,h). 
 $$
 Now we deal with the singular part. We have
 $$
  \int_Q \xi \Theta (u_n)\ d\mu_s^n =  \int_Q \xi \ d\mu_s^n+ \int_Q \xi( \Theta (u_n)-1)\ d\mu_s^n.
   $$
 Observe that $| \Theta (s)-1|\leq (h-s)^+$, so that 
 $$
 \begin{array}{l}
\dys \left|\int_Q \xi( \Theta (u_n)-1)\ d\mu_s^n\right|\\\\
\dys\leq \|\xi\|_{L^{\infty}(Q)} \left(\int_Q (h-u_n)^+\psi_\delta\ d\mu_s^n+\int_Q \ (1-\psi_\delta)d\mu_s^n\right)=\omega (n,\delta),
 \end{array}
 $$
 using respectively \rife{stima1} and Lemma \ref{cut}.  
 
 Finally, gathering together all these results we have 
$$
 \lim_{h}\limsup_{n} \int_{\{h-1\leq u_n<h\}} a(t,x,\nabla u_n)\nabla u_n \ \xi= \int_Q \xi \ d\mu_s\,.
$$
\end{proof}

In view of Proposition \ref{rico}, we proved that at least one renormalized solution, in the sense introduced above (Definition \ref{rendef}), satisfies the asymptotic property \rife{asy}. Let us stress that we actually expect such a property to  hold for {\em any} renormalized solution, although the proof might be technically quite involved. We refer the reader to   \cite[Proposition 4.9]{ppp} which contains several technical tricks for this purpose.

\subsection{Extension to the non-monotone case}

Our main existence result can be extended to a larger class of operators as for instance the ones involving a non-monotone dependence with respect to  $u$ in the function $a$. As an example,  we can consider a function $\tilde a : Q \times \re\times \rn \to \rn$ to be a Carath\'eodory function (i.e., $\tilde a (\cdot,\cdot,s,\xi)$
is measurable on $Q$ for every $(s,\xi)$ in $\re\times\rn$, and $\tilde a(t,x,\cdot,\cdot)$ is
continuous on $\re\times\rn$ for almost every $(t,x)$ in $Q$) such that the
following holds:
\be
\tilde a(t,x,s,\xi) \cdot \xi \geq \al|\xi|^p,
\label{coercp2}
\ee
\be
|\tilde a(t,x,s,\xi)| \leq \beta[b(t,x) +|s|^{p-1} +|\xi|^{p-1}],
\label{cont2}
\ee
\be
[\tilde a(t,x,s,\xi) - \tilde a(t,x,s,\eta)] \cdot (\xi - \eta) > 0,
\label{monot2}
\ee
for almost every $(t,x)$ in $Q$, for every $s\in\re$ and for every $\xi$, $\eta$ in $\rn$, with
$\xi \neq \eta$, where, as before, $p > 1$, 
$\al$ and $\beta$ are two positive constants, and
$b$ is a nonnegative function in $L^{p'}(Q)$. Notice that  \eqref{coercp2} ensure  that  $\tilde a(x,t,s,0)=0$ for any $s\in \re$ and a.e. $(t,x)\in Q$.

We can consider the parabolic problem, analogous to \rife{renpb}, associated to $\tilde a$, that is 
\begin{equation}\label{renpbnon}
\begin{cases}
    u_t-\dive(\tilde a(t,x,u {\nabla u} )=\mu & \text{in}\ Q,\\
    u(0,x)=u_0  & \text{in}\ \Omega,\\
 u(t,x)=0 &\text{on}\ (0,T)\times\partial\Omega,
  \end{cases}
\end{equation} 
where, as before, $\mu$ is  a nonnegative Radon measure on $Q$ such that $|\mu|(Q)<\infty$ and $u_0\in \elle1$ is a nonnegative function. The extension of Definition \ref{rendef} to this case is straightforward.

Existence of a renormalized solution for problem \rife{renpbnon} can be reproduced routinely by applying the capacitary estimates given in \cite[Theorem 6.1]{ppp}.

\end{document}